\documentclass[a4paper,10pt]{amsart}
\usepackage[english]{babel}
\usepackage{amsmath,amssymb,amsthm}
\usepackage[pdftex]{color}
\usepackage[bookmarks=true,hyperindex,pdftex,colorlinks, citecolor=MidnightBlue,linkcolor=MidnightBlue, urlcolor=MidnightBlue]{hyperref}
\usepackage[dvipsnames]{xcolor}
\usepackage{tikz-cd} 
\usepackage{tikz}
\usetikzlibrary{arrows}
\usepackage{longfbox}
\usepackage{enumerate}
\usepackage{lineno}
\usepackage[edges]{forest}
\usepackage{mathtools}

\usetikzlibrary{decorations.markings}
\tikzset{neg/.style={
		decoration={markings,
			mark= at position 0.5 with {
				\node[transform shape] (tempnode) {$\setminus$};
			}
		},
		postaction={decorate}
}}

\newtheorem{theorem}{Theorem}[section]

\newtheorem{lemma}[theorem]{Lemma}

\theoremstyle{definition}

\newtheorem*{definition*}{Definition}
\newtheorem*{proposition*}{Proposition}
\newtheorem*{theorem*}{Theorem}
\newtheorem*{corollary*}{Corollary}
\newtheorem*{example*}{Example}
\newtheorem*{problem*}{Problem}

\theoremstyle{remark}

\newcommand{\SA}{\operatorname{LipSNA}}

\newcommand{\Lip}{{\mathrm{Lip}}_0}

\begin{document}
	
	\title{Cantor sets of low density and Lipschitz functions on $C^1$ curves}
	
	\author[Chiclana, R.]{Rafael Chiclana}

	\address{Kent State University, Kent, Ohio}
	\email{rchiclan@kent.edu,}

	\begin{abstract}
		We characterize the functions $f\colon [0,1] \longrightarrow [0,1]$ for which there exists a measurable set $C\subseteq [0,1]$ of positive measure satisfying $\frac{|C\cap I|}{|I|}<f(|I|)$ for any nontrivial interval $I \subseteq [0,1]$. As an application, we prove that on any $C^1$ curve it is possible to construct a Lipschitz function that cannot be approximated by Lipschitz functions attaining their Lipschitz constant.
	\end{abstract}
	
	\maketitle
	
	\thispagestyle{plain}

	\section{Introduction}
	Let $|\cdot|$ denote the Lebesgue measure in $[0,1]$. Consider $C\subseteq[0,1]$ a measurable set with $|C|>0$ and $I\subseteq [0,1]$ an interval. In this paper we study how ``dense" the set $C$ is in $I$, which can be measured by
	\begin{equation}\label{ratio} \frac{|C\cap I|}{|I|}.
	\end{equation}
	More concretely, we study what conditions a function $f \colon [0,1] \longrightarrow [0,1]$ must satisfy to guarantee that there exists a measurable set $C\subseteq [0,1]$ with $|C|>0$ for which the ratio (\ref{ratio}) is bounded above by $f(|I|)$, for every nontrivial interval $I \subseteq [0,1]$. First, since $|C|>0$, it is clear that the infimum of $f$ must be positive. Moreover, Lebesgue's density theorem states that for almost every point $p \in C$ we have
	\[ \lim_{x\to 0^+} \frac{|C\cap [p-x,p+x]|}{2x} = 1,\]
	so $f(x)$ must converge to $1$ as $x$ approaches $0$. Our main result shows that these conditions are enough.
	\begin{theorem}\label{theo:Cantor}
		Let $f \colon [0,1] \longrightarrow \mathbb[0,1]$ be a function. Then, the following statements are equivalent:
		\begin{enumerate}[(i)]
			\item $\lim_{x\to 0^+} f(x)=1$ and $\inf_{[0,1]} f(x)>0$.
			\item There exists a measurable set $C\subseteq [0,1]$ of positive measure satisfying 
			\[ \frac{|C \cap I|}{|I|} < f(|I|) \quad \mbox{ for every nontrivial interval } I \subseteq [0,1].\]
		\end{enumerate}
	\end{theorem}
	In order to prove the above result, in Section \ref{section proof} we define the maximal density function of a measurable set (see (\ref{densityfunction})). Then, we construct a family of Cantor sets whose maximal density functions present a good behavior. The study of this behavior will be broken up into several lemmata. 
	
	Finally, in Section \ref{section3} we study Lipschitz functions defined on $C^1$ curves. Recall that given a metric space $(M,d)$, a function $f\colon M\longrightarrow \mathbb{R}$ is \textit{Lipschitz} if there is $k>0$ for which
	\begin{equation}\label{lipschitz} |f(p)-f(q)|\leq k d(p,q)  \quad \forall \, p,q \in M.
	\end{equation}
	The least constant satisfying (\ref{lipschitz}) is called the \textit{Lipschitz constant} of $f$, denoted by $\operatorname{L}(f)$, and is given by
	\begin{equation}\label{lipschitzconst} \operatorname{L}(f)=\sup\left \{\frac{|f(p)-f(q)|}{d(p,q)} \colon p\neq q \in M\right \}.
	\end{equation}
	We say that a Lipschitz function $f\colon M\longrightarrow \mathbb{R}$ \textit{attains its Lipschitz constant} if the supremum in (\ref{lipschitzconst}) is attained at a pair of distinct points $(p,q)$ of $M$. As an application of Theorem \ref{theo:Cantor}, Section \ref{section3} is devoted to prove the following result.
	
	\begin{theorem}\label{theo:C1} Let $E$ be a normed space, $J \subseteq \mathbb{R}$ an interval, $\alpha\colon J \longrightarrow E$ a $C^1$ curve with $\alpha'$ nonidentically zero, and $\Gamma \subseteq E$ its range. Then, there exists a Lipschitz function $H\colon \Gamma \longrightarrow \mathbb{R}$ that cannot be approximated by Lipschitz functions from $\Gamma$ to $\mathbb{R}$ attaining their Lipschitz constant.
	\end{theorem}
	
	The meaning of "approximated" in the above result is explained in more detail in Section \ref{section3}, where we give a little background in Lipschitz functions. Theorem \ref{theo:C1} generalizes Theorem 2.1 in \cite{chiclana2021examples}, where the same result is shown for the unit circumference in $\mathbb{R}^2$ endowed with the Euclidean metric, and provides many new examples of metric spaces for which not every Lipschitz function can be approximated by Lipschitz functions attaining their Lipschitz constant. Moreover, it gives a partial answer for Remark 2.4 in \cite{chiclana2021examples}, that asks whether there is a distance $d'$ on $[0,1]$ equivalent to the usual one, such that every Lipschitz function from $([0,1],d')$ to $\mathbb{R}$ can be approximated by Lipschitz functions attaining their Lipschitz constant. Theorem \ref{theo:C1} shows that this is not the case when the distance $d'$ comes from a $C^1$ curve.
	
	\section{Proof of the main result}\label{section proof}
	
	Let $C \subseteq[0,1]$ be a measurable set of positive measure. Notice that the value of the ratio (\ref{ratio}) for two intervals $I$ and $I'$ of the same length may be different. In order to get an upper bound in terms of the length of the intervals, let the \textit{maximal density function} of $C$ be the function $\phi_C\colon (0,1] \longrightarrow \mathbb{R}$ defined by
	\begin{equation}\label{densityfunction} \phi_C(s)=\sup\left \{\frac{|C\cap I|}{|I|} \colon I\subseteq [0,1] \mbox{ is an interval of length } s\right \} \quad \forall \, s \in (0,1].
	\end{equation}
	Consider a sequence of real numbers $\{\lambda_n\}_{n\in\mathbb{N}}$ with $0<\lambda_n<1$ for every $n \in \mathbb{N}$. Associated to this sequence, we are going to construct a Cantor set. Consider $C_0=[0,1]$. Divide $C_0$ into two pieces: $[0,\frac12]$ and $[\frac12,1]$, and consider in each of them, starting from the left, two intervals $I_{1,1}$, $I_{1,2}$, of length $\frac{|C_0|}{2}(1-\lambda_1)$. Then, define $C_1=I_{1,1}\cup I_{1,2}$, that is, $C_1=[0,\frac{1}{2}(1-\lambda_1)]\cup[\frac{1}{2},\frac{1}{2}+\frac{1}{2}(1-\lambda_1)]$. Now, divide each connected component of $C_1$ into two new pieces of the same length and consider in each of them, starting from the left, two intervals of length $\frac{|C_1|}{4}(1-\lambda_2)$. Then, define $C_2$ to be the union of the new intervals that we have constructed. Repeating this process, we construct $C_n \subseteq [0,1]$ as a finite union of closed intervals, for every $n \in \mathbb{N}$. Finally, we define the Cantor set associated to $\{\lambda_n\}_{n\in\mathbb{N}}$ by $$C=\bigcap_{n\in \mathbb{N}} C_n.$$
	See the figure below to get an idea of the shape of these Cantor sets.
	
	\begin{figure}[htb]
		\centering
		\includegraphics[scale=12]{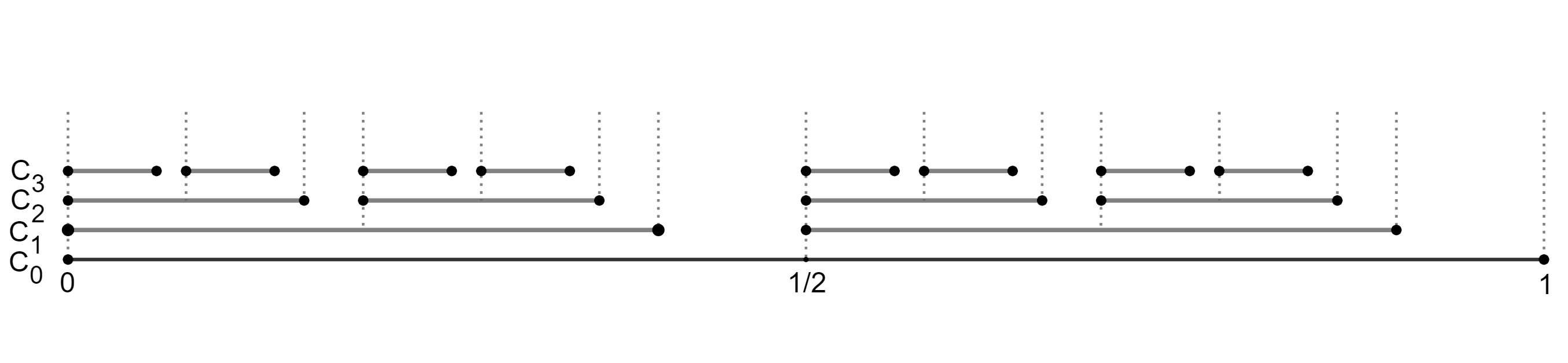}
		\caption{Shape of the Cantor set $C$ on its first levels.}
	\end{figure}
	
	We need some notation to study the structure of these Cantor sets. For any $n \in \mathbb{N}\cup\{0\}$, $C_n$ is the union of $2^n$ closed intervals of the same length. Denote $I_n$ the first interval composing $C_n$. Also, write $r_n$ for the length of $I_n$ and $g_n$ for the length of the gap that we find right after the interval $I_n$ in $C_n$. The following lemma gives us some measurements that we will need.
	
	\begin{lemma}\label{lemma:1}
		Given $\{\lambda_n\}_{n \in \mathbb{N}}\subseteq (0,1)$, let $C$ be the Cantor set associated to $\{\lambda_n\}_{n \in \mathbb{N}}$. Then,
		\begin{enumerate}[(i)]
			\item $r_n=\frac{1}{2^n}\prod_{i=1}^n(1-\lambda_i)$ for every $n \in \mathbb{N}$.
			\item $|C|=\prod_{i=1}^\infty (1-\lambda_i).$
			\item $g_n=\frac{1}{2}r_{n-1}\lambda_n$ for every $n \in \mathbb{N}$.
		\end{enumerate}
	\end{lemma}
	
	\begin{proof}
		To prove (i), we work by induction on $n$. For $n=1$ we have $r_1=\frac{1}{2}(1-\lambda_1)$ by construction. Now, suppose the statement true for $n \in \mathbb{N}$. By definition, to obtain $I_{n+1}$ we divide $I_n$ into two pieces of the same length, and then $I_{n+1}$ is the interval starting at $0$ of length $\frac{|I_n|}{2}(1-\lambda_{n+1})$. Therefore,
		\[r_{n+1}=|I_{n+1}|=\frac{|I_n|}{2}(1-\lambda_{n+1})=\frac{r_n}{2}(1-\lambda_{n+1})=\frac{1}{2^{n+1}}\prod_{i=1}^{n+1}(1-\lambda_i).\]
		Next, notice that $C_{n}$ is composed of $2^{n}$ disjoint closed intervals of the same length. Hence,
		\[ |C_n|=2^n r_n=\prod_{i=1}^n (1-\lambda_i) \quad \forall \, n \in \mathbb{N}.\]
		We just need to take limit as $n$ goes to infinity in the above equality to obtain (ii). Finally, notice that $r_{n-1}=2r_n+2g_n$ for any $n \in \mathbb{N}$, from where we deduce that
		\begin{align*} g_n&=\frac{1}{2}(r_{n-1}-2r_n)=\frac{1}{2}\left(\frac{1}{2^{n-1}}\prod_{i=1}^{n-1}(1-\lambda_i)-\frac{2}{2^n}\prod_{i=1}^n(1-\lambda_i)\right)=\frac{1}{2}r_{n-1}\lambda_n.\qedhere
		\end{align*}
	\end{proof}

	The interesting case is when the Cantor set $C$ associated to $\{\lambda_n\}_{n\in\mathbb{N}}$ has positive measure. In view of Lemma \ref{lemma:1}, this will happen when $\prod_{n=1}^\infty (1-\lambda_n) >0$, or equivalently, $\sum_{n=1}^\infty \lambda_n < \infty$. 
	
	The next lemma gives an explicit formula for the maximal density function of $C$.
	
	\begin{lemma}\label{lemma:s} Given $\{\lambda_n\}_{n \in \mathbb{N}}\subseteq (0,1)$, let $C$ be the Cantor set associated to $\{\lambda_n\}_{n \in \mathbb{N}}$. Then,
		\[ \phi_C(s)=\frac{|C\cap [0,s]|}{s} \quad \forall \, s \in (0,1].\]
		Consequently, $\phi_C$ is a continuous function.
	\end{lemma}
	
	\begin{proof}
		If $s=1$, there is nothing to prove. Fix $s\in (0,1)$ and consider $n \in \mathbb{N}$ so that $r_n\leq s < r_{n-1}$. We want to show that the supremum in (\ref{densityfunction}) is attained at the interval $[0,s]$, that is,
		\[ |C\cap I|\leq|C\cap [0,s]| \quad \mbox{ for any interval } I \subseteq [0,1] \mbox{ of length } s.\]
		Claim: Given $I\subseteq[0,1]$ an interval of length $s<1$, pick $n \in \mathbb{N}$ such that $r_n\leq s <r_{n-1}$. Then, there is an interval $I^*$ of length $s^*<r_n$ satisfying
		\[ |C\cap I|-|C\cap [0,s]| \leq |C\cap I^*| - |C\cap [0,s^*]|.\]
		If the claim is true, then we can apply it $k$ times to obtain an interval $I^*_k$ of length $s^*_k<r_{n+k-1}$ satisfying
		\[ |C\cap I|-|C\cap [0,s]| \leq |C\cap I_k^*| - |C\cap [0,s^*_k]|\leq|C\cap I_k^*|\leq r_{n+k-1} .\]
		Since $\{r_n\}_{n \in \mathbb{N}}$ converges to $0$, we conclude that $|C\cap I|\leq |C\cap [0,s]|$.
		
		Let us prove the claim. We distinguish four cases:
		
		\begin{itemize}
			\item Case 1: $I$ does not intersect any connected component of $C_n$. This implies that $C\cap I=\emptyset$ and so $|C\cap I|=0$. Hence, we can take $I^*$ any interval of length smaller than $r_n$.
			\item Case 2: $I$ intersects exactly one connected component of $C_n$. Write $J$ for that connected component. Then, $I\setminus J$ lies on a gap. Recall that $I_n=[0,r_n]$ and $s\geq r_n$, which implies that
			\[ |C\cap I|=|C\cap(I\cap J)|\leq |C\cap J|=|C\cap I_n|\leq |C\cap [0,s]|.\]
			Therefore, $|C\cap I|-|C\cap[0,s]|\leq 0$, so we can set $I^*=I_{n+1}$.
			\item Case 3: $I$ intersects exactly two connected components of $C_n$. Going from left to right, denote such components by $J_1$ and $J_2$. Also, let us write $I\cap J_1=[a_1,b_1]$ and $I\cap J_2=[a_2,b_2]$. The shape of $C$ allows to identify $I\cap J_2$ with a subinterval of $I_n$ starting at $0$. More precisely,
			\[ |C\cap (I\cap J_2)|=|C\cap[a_2,b_2]|=|C\cap[0,b_2-a_2]|,\]
			Analogously, we can identify $I\cap J_1$ with a subinterval of $I_n$ ending at $r_n$. Indeed,
			\[ |C\cap (I\cap J_1)|=|C\cap[a_1,b_1]|=|C\cap [r_n-(b_1-a_1),r_n]|.\]
			Write $\ell=(b_1-a_1)+(b_2-a_2)$. If $\ell\leq r_n$, then we have that $I\cap J_1$ and $I\cap J_2$ can be identified with two subintervals of $I_n$ whose intersection is either empty or a single point. Therefore,
			\begin{align*}
				|C\cap I|&= |C\cap(I\cap J_1)|+ |C\cap(I\cap J_2)|=
				|C\cap [0,b_2-a_2]|+|C\cap[r_n-(b_1-a_1),r_n]|\\
				&\leq |C\cap I_n|\leq |C\cap[0,s]|,
			\end{align*}
			from where we get $|C\cap I|-|C\cap[0,s]|\leq 0$, so we can set $I^*=I_{n+1}$. If $\ell>r_n$, write
			\[I_0^*=[0,b_2-a_2] \cap [r_n-(b_1-a_1),r_n] = [r_n-(b_1-a_1),b_2-a_2]. \]
			Set $s^* =\ell-r_n$ to be the length of the interval $I_0^*$. Since $s \geq \ell+g_n$, we have $s-(r_n+g_n)\geq s^*$.
			If $s^*\geq r_n$ then we would have that $s - (r_n+g_n)\geq r_n$, and so $s\geq 2r_n+g_n$. However, this implies
			\[ |C\cap [0,s]|\geq2|C\cap I_n|\geq |C\cap I|.\]
			Therefore, $|C\cap I|-|C\cap [0,s]|\leq 0$, so we can set $I^*=I_{n+1}$. Consequently, we may suppose $s^*<r_n$. One the one hand, notice that
			\begin{align*}
				|C\cap I|&=|C\cap(I\cap J_1)|+ |C\cap(I\cap J_2)|
				=|C\cap I_n|+|C\cap I^*_0|.
			\end{align*}
			On the other hand,
			\[ |C\cap [0,s]|=|C\cap I_n|+|C\cap [r_n+g_n,s]|=|C\cap I_n|+|C\cap [0,s-(r_n+g_n)]|. \]
			Moreover, since $s-(r_n+g_n)\geq s^*$, we have $|C \cap [0,s]|\geq |C\cap I_n|+|C\cap[0,s^*]|$, that yields
			\[ |C\cap I|-|C\cap [0,s]|\leq |C\cap I_0^*| - |C\cap [0,s^*]|.\]
			Therefore, we can set $I^*=I^*_0$.
			\item Case 4: $I$ intersects three or more connected components of $C_n$. First, $I$ cannot intersect four connected components, since in such case we would have
			\[ s\geq 2r_n+2g_n=r_{n-1}.\] Therefore, $I$ intersects exactly three connected components of $C_n$. Going from left to right, denote such components by $J_1$, $J_2$, and $J_3$. If the intersection with any of them is just a point, then in terms of Lebesgue measure, we may assume that $I$ actually intersects only two intervals, case that we have already studied. Write $I\cap J_1=[a_1,b_1]$, $I\cap J_3=[a_3,b_3]$. Then, we have
			\begin{align*}|C\cap I|&=|C\cap (I\cap J_1)|+|C\cap (I\cap J_2)|+|C\cap (I\cap J_3)|\\
				&=|C\cap (I\cap J_1)|+|C\cap I_n|+|C\cap (I\cap J_3)|.
			\end{align*}
			Let us write $\ell=(b_1-a_1)+(b_3-a_3)$. Then we must have that $\ell<r_n$. In fact, notice that
			\begin{equation}\label{l2eq1} s\geq (b_1-a_1) + g_n + r_n + g_n + (b_3-a_3)=\ell+r_n+2g_n.
			\end{equation}
			If we suppose $\ell\geq r_n$, then $s\geq 2r_n+2g_n=r_{n-1}$, a contradiction. Now, identify $I\cap J_3$ with a subinterval of $[r_n+g_n,2r_n+g_n]$ of the same length as $I\cap J_3$ starting at $r_n+g_n$. More precisely,
			\[ |C\cap (I\cap J_3)|=|C\cap[a_3,b_3]|=|C\cap [r_n+g_n,r_n+g_n + (b_3-a_3)]|,\]
			Also, we can identify $I \cap J_1$ with a subinterval of $I_n$ ending at $r_n$ of the same length as $I\cap J_1$.
			\begin{align*}|C\cap (I\cap J_1)|&=|C\cap [a_1,b_1]|=|C\cap [r_n-(b_1-a_1),r_n]|.
			\end{align*}
			Then, if we consider a new interval $I^*=[r_n-(b_1-a_1),r_n+g_n+(b_3-a_3)]$ we will have that $|C\cap (I\cap J_1)|+|C\cap (I\cap J_3)|= |C\cap I^*|$, and so
			\begin{equation}\label{l2eq2}
				|C\cap I|=|C\cap I_n|+|C\cap I^*|.
			\end{equation}
			Furthermore, since $I$ intersects three connected components of $C_n$, we have $s\geq r_n+g_n$. Thus,
			\[ |C\cap [0,s]|= |C\cap I_n|+|C \cap [r_n+g_n,s]= |C\cap I_n|+ |C\cap [0,s-(r_n+g_n)]|.\]
			We claim that $s-(r_n+g_n)\geq |I^*|$. Indeed, if we suppose $s-(r_n+g_n)< |I^*|$ we would have
			\[ s-(r_n+g_n)< (r_n+g_n+(b_3-a_3))-(r_n-(b_1-a_1)),\]
			from where we obtain $s<r_n+2g_n+\ell$, which contradicts (\ref{l2eq1}). Hence,
			\begin{equation}\label{l2eq3} |C\cap [0,s]|\geq |C\cap I_n| + |C \cap [0,|I^*|]|.
			\end{equation}
			Consequently, (\ref{l2eq2}) and (\ref{l2eq3}) gives
			\[ |C\cap I|-|C\cap [0,s]|\leq |C\cap I^*| - |C \cap [0,|I^*|]|.\]
			Since $I^*$ intersects only two connected components of $C_n$, this case reduces to case 3.\qedhere
		\end{itemize}
	\end{proof}
	
	 We will also need the following lemma.
	
	\begin{lemma}\label{lemma:4}
		Given a decreasing sequence $\{\lambda_n\}_{n\in \mathbb{N}}\subseteq (0,1)$, let $C$ be the Cantor set associated to $\{\lambda_n\}_{n\in\mathbb{N}}$. Pick $s \in (0,1)$, and take $n \in \mathbb{N}$ so that $r_n\leq s < r_{n-1}$. Then, $$\phi_C(s)\leq \phi_C(r_n).$$
	\end{lemma}
	
	\begin{proof}
		Pick $s \in (0,1)$, and take $n \in \mathbb{N}$ so that $r_n\leq s < r_{n-1}$. In view of Lemma \ref{lemma:s}, we need to prove
		\[ \frac{|C\cap[0,s]|}{s} \leq \frac{|C\cap[0,r_n]|}{r_n}.\]
		Since $[0,1]\setminus C$ is dense in $[0,1]$ and $\phi_C$ is a continuous function, we may assume that $s\notin C$. Take the smallest integer $k \in \mathbb{N}$ so that $s \notin C_k$. Clearly, we must have that $k\geq n$. Moreover, since $I_n$ is formed by $2^{k-n}$ copies of $I_k$ and gaps, we have
		\begin{equation}\label{eq1} |C\cap I_n|=2^{k-n}|C\cap I_k|.
		\end{equation}
		Furthermore, since $r_{n+i}=2r_{n+i+1}+2g_{n+i+1}$ for every $i \in \mathbb{N}$, by induction we obtain that
		\begin{equation}\label{eq1.1} r_{n+i}=2^{k-(n+i)}r_k+\sum_{j=i+1}^{k-n} 2^{j-i}g_{n+j}=2^{k-(n+i)}r_k+ G_{n+i}  \quad \forall \, i\in \{0,\ldots,k-n\},
		\end{equation}
		where $G_{n+i} = \sum_{j=i+1}^{k-n} 2^{j-i}g_{n+j}$ represents the measure of all gaps in $I_{n+i}$ of order less or equal than $k$, understanding $G_{k}=0$.
		On the other hand, we know that $s$ lies on a gap of $C_k$. Going from left to right, let $J=[a,b]$ be the interval of $C_k$ that we find right before such a gap. Then, $|C\cap [0,s]|=|C\cap[0,b]|$ and clearly $s\geq b$, so $\phi_C(s)\leq\phi_C(b)$. Therefore, it will be enough to show that 
		\begin{equation}\label{eq1.5}
			\frac{|C\cap [0,b]|}{b} \leq \frac{|C\cap I_n|}{r_n}.
		\end{equation}
		Since $b$ is the end point of some interval of $C_k$, we can write $b$ as
		\begin{equation}\label{eq2} b = r_k+ \sum_{i=0}^{k-n} \theta_i(r_{n+i} + g_{n+i})
		\end{equation}
		for some $\theta_i \in \{0,1\}$ for $i=0,\ldots,k-n$. We may suppose $b> r_n$, which implies $\theta_0=1$. Using (\ref{eq1.1}) we can decompose $b$ into a sum of copies of $I_k$ and gaps.
		\begin{align*}
			b &= r_k+ \sum_{i=0}^{k-n} \theta_i(r_{n+i} + g_{n+i}) = r_k+\sum_{i=0}^{k-n} \theta_i\left (2^{k-(n+i)}r_k+ G_{n+i}+ g_{n+i}\right )
			\\&=r_k\left (1+\sum_{i=0}^{k-n}\theta_i 2^{k-(n+i)}\right) + \sum_{i=0}^{k-n}\theta_i(G_{n+i}+ g_{n+i}).
		\end{align*}
		Consequently,
		\begin{equation}\label{eq3} |C\cap [0,b]|= |C\cap I_k| \left (1+\sum_{i=0}^{k-n}\theta_i 2^{k-(n+i)}\right).
		\end{equation}
		Now, we can substitute in (\ref{eq1.5}) using equalities (\ref{eq1}), (\ref{eq1.1}), (\ref{eq3}), and the decomposition of $b$. After simplifying the obtained expression, we can rewrite inequality (\ref{eq1.5}) as
		\[ G_n\left (1+\sum_{i=0}^{k-n}\theta_i 2^{k-(n+i)}\right)\leq 2^{k-n} \sum_{i=0}^{k-n}\theta_i(G_{n+i}+ g_{n+i}).\]
		Equivalently, it will be enough to show that
		\begin{equation}\label{eq4} \sum_{i=0}^{k-n} 2^{k-(n+i)} \theta_i( 2^i (G_{n+i}+g_{n+i}) - G_n)\geq G_n.
		\end{equation}
		Notice that $G_i=2G_{i+1} + 2g_{i+1}$ for every $i\in\{n,\ldots,k\}$. Therefore,
		\[ G_n=2^iG_{n+i} +\sum_{j=1}^i 2^j g_{n+j}, \quad \forall \, i \in \{1,\ldots,k-n\}.\]
		From here, for any $i \in \{1,\ldots,k-n\}$ we deduce that
		\begin{align*}
			2^i(G_{n+i} + g_{n+i}) -G_n=2^ig_{n+i} - \sum_{j=1}^i 2^jg_{n+j}=-\sum_{j=1}^{i-1} 2^jg_{n+j}.
		\end{align*}
		Therefore, inequality (\ref{eq4}) can be rewritten as
		\[ 2^{k-n} g_n - \sum_{i=1}^{k-n} \left( 2^{k-(n+i)}\theta_i \left (\sum_{j=1}^{i-j} 2^jg_{n+j}\right )\right )\geq G_n.\]
		In view of this, it will be enough to study the case when $\theta_i=1$ for every $i=0,\ldots,k-n$. In such case, we deduce from (\ref{eq1}) and (\ref{eq3}) that $|C\cap [0,b]|=2|C\cap I_{n}|$. Finally, we claim that $b\geq 2r_n$, from where (\ref{eq1.5}) follows. Indeed, since $\lambda_{i+1}\leq \lambda_i$  for every $i \in \mathbb{N}$, we get
		\[ \frac{g_{i+1}}{g_{i}} =\frac{\lambda_{i+1}}{\lambda_{i}}\frac{r_{i}}{r_{i-1}}\leq \frac{r_i}{r_{i-1}}=\frac{\frac{1}{2^{i}}\prod_{j=1}^{i}(1-\lambda_j)}{\frac{1}{2^{i-1}}\prod_{j=1}^{i-1}(1-\lambda_j)}=\frac{1}{2}(1-\lambda_i)<\frac{1}{2}.\]
		As a consequence of this, we deduce that
		\[ \sum_{i=1}^{k-n} g_{n+i} \leq\sum_{i=1}^\infty g_{n+i} \leq \sum_{i=1}^\infty \frac{1}{2^i} g_{n}=g_n. \]
		Since we are assuming that $b$ is the end point of the last interval composing $C_k$, we conclude that
		\[ b=r_{n-1} - \sum_{i=0}^{k-n} g_{n+i}= 2r_n+2g_n -g_n- \sum_{i=1}^{k-n} g_{n+i} \geq 2r_n.\qedhere\]
	\end{proof}
	
	We are now able to present the proof of the main result.
	
	\begin{proof}[Proof of Theorem \ref{theo:Cantor}]
		As we commented in the introduction, (ii)$\Rightarrow$ (i) easily follows from Lebesgue's density theorem and the fact that $|C|>0$. Thus, it remains to prove (i) $\Rightarrow$ (ii). First, let us assume that $f$ is decreasing. We will prove that there exists a decreasing sequence of numbers $\{\lambda_n\}_{n \in \mathbb{N}}\subseteq (0,1)$ for which its associated Cantor set $C$ satisfies $|C|>0$ and $\phi_C(r_n)<f(2^{-n+1})$ for every $n \in \mathbb{N}$, or equivalently,
		\begin{equation}\label{conditionlambda} \prod_{j=n+1}^\infty (1-\lambda_j) < f(2^{-n+1}) \quad \forall \, n \in \mathbb{N}.
		\end{equation}
		Taking logarithm and multiplying by $-1$ gives the following equivalent condition.
		\[ \sum_{j=n+1}^\infty \log\left (\frac{1}{1-\lambda_j}\right ) > -\log( f(2^{-n+1}))\quad \forall \, n \in \mathbb{N}.\]
		Notice that by hypothesis $\inf_{[0,1]} f(x)=f(1)>0$, so the right term in above expression makes sense. Write $L_n = -\log( f(2^{-n+1}))$ for every $n \in \mathbb{N}$. Then, since $\lim_{x\to 0^+} f(x)=1$,  $\{L_n\}_{n \in \mathbb{N}}$ is a decreasing sequence of positive numbers converging to zero. It is not difficult to see that there exists another decreasing sequence of positive numbers $\{ \ell_n \}$ satisfying that $\sum_{j=0}^\infty \ell_j <\infty$ and
		\[ \sum_{j=n+1}^\infty \ell_j > L_n \quad \forall \, n \in \mathbb{N}.\]
		We just need to take $\lambda_n= 1-e^{-\ell_n}$ for every $n\in \mathbb{N}$ to obtain a decreasing sequence in $(0,1)$ for which $|C|>0$ and condition (\ref{conditionlambda}) is satisfied. Since $f$ is decreasing and $r_{n-1} \leq 2^{-n+1}$, we also obtain that $\phi_C(r_n)<f(r_{n-1})$ for every $n \in \mathbb{N}$. Now, pick $s \in (0,1]$ and take $n \in \mathbb{N}$ so that $r_n< s \leq r_{n-1}$. Then,
		\[ \phi_{C}(s)\leq \phi_{C}(r_n)<f(r_{n-1})\leq f(s),\]
		where the first inequality comes from Lemma \ref{lemma:4} and the last inequality comes from the fact that $f$ is decreasing. The result follows from the definition of the maximal density function (see (\ref{densityfunction})).
		
		We now proceed to prove the general case. Let $g$ be a function satisfying (i). Define $h\colon[0,1]\longrightarrow \mathbb{R}$ by
		\[ h(s)=\inf \{g(x) \colon x\in [0,s]\} \quad \forall \, s \in [0,1].\]
		Then, $h$ is decreasing and satisfies (i). Indeed, it is clear that $\lim_{s\to 0^+} h(s)=\lim_{s \to 0^+} g(s)=1$. Also,
		\[ \inf\{h(s)\colon s \in [0,1]\}=h(1)=\inf\{g(s)\colon s \in [0,1]\}>0.\]	
		Hence, applying the previous case to the function $h$ gives a Cantor set $C\subseteq [0,1]$ of positive measure so that $\phi_{C}(s)<h(s)$ for every $s \in(0,1]$. Finally, notice that $h(s)\leq g(s)$ for every $s \in [0,1]$.
	\end{proof}

	\section{Lipschitz maps on $C^1$ curves}\label{section3}
	Consider a \textit{pointed} metric space $M$, that is, we distinguish in $M$ a point that we denote by $0$. The map $\|\cdot\|_L$ that assigns to every Lipschitz function $f$ its Lipschitz constant $\operatorname{L}(f)$ is a complete norm in the space of Lipschitz functions from $M$ to $\mathbb{R}$ vanishing at $0$. This Banach space is usually called the \textit{Lipschitz space} over $M$ and is denoted by $\Lip(M,\mathbb{R})$. The study of Lipschitz spaces has been strongly developed during the last decades (good references for background are \cite{godefroy2015survey} and \cite{weaver2018lipschitz}). Moreover, this line of research recently has gained popularity due to the importance that Lipschitz maps have in the theory of nonlinear geometry of Banach spaces. Furthermore, the study of Lipschitz-free spaces, which are isometric preduals of the Lipschitz spaces, is a very active line of research nowadays (see \cite{albiac2021lipschitz},\cite{aliaga2021purely}, \cite{cuth2016on}, \cite{godard2010tree}, and \cite{godefroy2003lipschitz} for instance).
	
	Write $\SA(M,\mathbb{R})$ for the set of Lipschitz functions from $M$ to $\mathbb{R}$ attaining its Lipschitz constant. In the survey paper \cite{godefroy2015survey}, it is asked the following natural question: for which metric spaces is it possible to approximate every Lipschitz function by Lipschitz functions attaining their Lipschitz constant, or equivalently, $\SA(M,\mathbb{R})$ is norm-dense in $\Lip(M,\mathbb{R})$? For background see \cite{cascales2019on}, where this problem is deeply studied. 
	
	The first negative result appeared in \cite{kadets2016normattaining}, where the authors proved that if we consider $[0,1]$ with the usual metric, then $\SA([0,1],\mathbb{R})$ is not dense in $\Lip([0,1],\mathbb{R})$. In some sense, density fails for this metric space because all points are metrically aligned, that is, the triangle inequality is always an equality. After this result, some generalizations appeared to give new examples of metric spaces $M$ for which $\SA(M,\mathbb{R})$ is not dense in $\Lip(M,\mathbb{R})$. However, all these new examples are similar in the sense that we find many points metrically aligned (or almost metrically aligned) in them. The only negative example which differs from them is Theorem 2.1 in \cite{chiclana2021examples}, which shows that $\SA(T,\mathbb{R})$ is not dense in $\Lip(T,\mathbb{R})$, where $T$ denotes the unit circumference in $\mathbb{R}^2$, endowed with the Euclidean metric. This section is devoted to prove Theorem \ref{theo:C1}, which generalizes the previous result for $C^1$ curves. In order to prove Theorem \ref{theo:C1}, we will apply Theorem \ref{theo:Cantor} to a specific function. The following preliminary result shows that such a function satisfies one of the necessary conditions of Theorem \ref{theo:Cantor}.
	
	\begin{lemma}\label{lemma:inf}
		Let $E$ be a normed space, $\rho>0$, and $\alpha \colon [0,\rho] \longrightarrow E$ a $C^1$ curve parametrized by arc length. Consider the function $g\colon (0,\rho] \longrightarrow \mathbb{R}$ given by
		\[ g(x)=\inf \left \{\frac{\|\alpha(t)-\alpha(s)\|}{|t-s|} \colon t,s \in [0,\rho], |t-s|=x\right \} \quad \forall \, x \in (0,\rho].\]
		Then, $\lim_{x\to 0^+} g(x)=1$.
	\end{lemma}
	
	\begin{proof} 
		First, it is clear that $|g(x)|\leq 1$ for every $x \in (0,\rho]$. Now, assume the statement is not true. Then, we find $\varepsilon>0$ and sequences $\{t_n\}$, $\{s_n\} \subseteq [0,\rho]$ with $t_n\neq s_n$ so that $\lim_{n\to\infty}|t_n-s_n|=0$ and
		\[ \frac{\|\alpha(t_n)-\alpha(s_n)\|}{|t_n-s_n|} <1-\varepsilon \quad \forall \, n \in \mathbb{N}.\]
		However, this contradicts the fact that $\|\alpha'(t)\|=1$ for every $t \in [0,\rho]$. Indeed, since $[0,\rho]$ is compact and $|t_n-s_n|$ goes to $0$, we find partial sequences $\{t_{\sigma(n)}\}$, $\{s_{\sigma(n)}\}$ converging to a common point $t_0\in [0,\rho]$. On the other hand, it is a standard fact that the map $\Phi \colon [0,\rho] \longrightarrow E$ given by
		\[ \Phi (t,s) = \left \{ \begin{array}{ll}
			\frac{\alpha(t)-\alpha(s)}{t-s} & \mbox{ if $t,s \in [0,\rho]$ with $ t\neq s$};\\
			\alpha'(t) & \mbox{ if $t,s \in [0,\rho]$ with $ t=s$},
		\end{array} \right.\]
		is continuous. This easily follows from the fact that that $\alpha$ is $C^1$ and Taylor's formula. Consequently, $\|\Phi(\cdot,\cdot)\|$ is also continuous, but $\|\Phi(t_0,t_0)\|=\|\alpha'(t_0)\|=1$, so we must have
		\[ \lim_{n \to \infty}
		\frac{\|\alpha(t_{\sigma(n)}) - \alpha(s_{\sigma(n)})\|}{|t_n-s_n|} = 1,\]
		which contradicts the assumption.
	\end{proof}
	
	Let $E$ be a normed space and $\alpha\colon [0,r] \longrightarrow E$ a $C^1$ curve parametrized by arc length, for some $r>0$. Then, there is $0<\rho<1$ small enough so that
	\begin{equation}\label{equivdist} \frac{1}{2}|t-s|\leq \|\alpha(t)-\alpha(s)\|\leq|t-s| \quad \forall \, t,s \in [0,\rho].
	\end{equation}
	Consequently, the function $g$ considered in Lemma \ref{lemma:inf} satisfies $\inf_{[0,\rho]} g(x)>0$. We can extend $g$ to $[0,1]$ with $g(x)=g(\rho)$ for every $\rho\leq x \leq 1$ to get a function to which Theorem \ref{theo:Cantor} applies.
	
	Given a Lipschitz function $f \colon M \longrightarrow \mathbb{R}$ and two distinct points $p$, $q \in M$, we will use $f(m_{p,q})$ to denote the quotient $\frac{f(p)-f(q)}{d(p,q)}$. We say that a point $p_0 \in M$ is \textit{norming} for the function $f$ if there are two sequences $\{p_n\}_{n \in \mathbb{N}}$, $\{q_n\}_{n\in \mathbb{N}}\subseteq M$, with $p_n\neq q_n$ for every $n \in \mathbb{N}$, satisfying
	\[ \quad \lim_{n \to \infty} p_n= \lim_{n \to \infty} q_n = p_0 \quad \mbox{ and } \quad \lim_{n \to \infty} \|f(m_{p_n,q_n})\| = \|f\|_L.\]
	Recall that every nontrivial $C^1$ curve can be locally reparametrized with respect to arc length. The following lemma allows us to restrict the study of the curve to a small arc.

	\begin{lemma}\label{lemma:open}
		Let $M$ be a metric space and let $N$ be a subset of $M$ with nonempty interior. Assume that there exists $f_0 \in \Lip(N,\mathbb{R})\setminus \overline{\SA(N,\mathbb{R})}$ with a norming point $p_0$ that belongs to the interior of $N$. Then, $\SA(M,\mathbb{R})$ is not dense in $\Lip(M,\mathbb{R})$.
	\end{lemma}
	
	\begin{proof}
		We may and do assume that the distinguished point of $N$ and $M$ is $p_0$. Let $f_0 \in \Lip(N,\mathbb{R})$ be the function given by the hypothesis.
		Pick $\delta>0$ such that for every $g_0 \in \Lip(N,\mathbb{R})$ satisfying $\|f_0-g_0\|_L<\delta$, we have $g_0 \notin \overline{\SA(N,\mathbb{R})}$. Consider $r_0 \in \mathbb{R}^+$ such that $B(p_0,r_0)\subseteq N$ and pick $0<\varepsilon<\frac{12\delta}{5r_0}$. Let us define a function $\varphi \colon M \longrightarrow \mathbb{R}$ by
		\[ \varphi(p) = \left \{ \begin{array}{ll}
			1 & \mbox{ if $p \in B(p_0,\frac{r_0}{4})$};\\
			1-\varepsilon(d(p,p_0)-\frac{r_0}{4}) & \mbox{ if $p \in B(p_0,\frac{r_0}{3})\setminus B(p_0,\frac{r_0}{4})$};\\
			1-\frac{\varepsilon r_0}{12} & \mbox{ if $p \notin B(p_0,\frac{r_0}{3})$}.
		\end{array} \right.\]
		It is immediate to verify that $\varphi$ is Lipschitz and $\|\varphi\|_L\leq \varepsilon$. Now, let $f\colon M \longrightarrow \mathbb{R}$ be an extension of $f_0$ via McShane preserving the norm $\|\cdot\|_L$. Define $g \colon M \longrightarrow \mathbb{R}$ by $g(p)=f(p) \varphi(p)$ for every $p \in M$. It is clear that $g$ is Lipschitz. In fact, if $p \in B(p_0,\frac{r_0}{3})$ and $q \in M$, notice that
		
		\begin{align*} |(g-f)(m_{p,q})|&=\frac{|(g(p)-f(p))-(g(q)-f(q))|}{d(p,q)}\\
			&=\frac{|(\varphi(p)-\varphi(q))f(p) + (\varphi(q)f(p)-f(p)) - (\varphi(q)f(q)-f(q))|}{d(p,q)}\\
			&\leq |f(p)||\varphi(m_{p,q})| + \frac{| (f(p)-f(q))-(\varphi(q)f(p)-\varphi(q)f(q))|}{d(p,q)}\\
			&= |f(p)||\varphi(m_{p,q})| +  \frac{|(f(p)-f(q))(1-\varphi(q))|}{d(p,q)}
			\leq \frac{r_0}{3}\varepsilon+\frac{\varepsilon r_0}{12}=\frac{5r_0}{12}\varepsilon<\delta.
		\end{align*}
		By symmetry, the same happens if we assume $p \in M$, $q\in B(p_0,\frac{r_0}{3})$. Finally, if neither $p$ nor $q$ belong to $B(p_0,\frac{r_0}{3})$, then we simply have that
		\[ |(g-f)(m_{p,q})|=\left |\left (1-\frac{\varepsilon r_0}{12}\right )f(m_{p,q}) - f(m_{p,q})\right |\leq \frac{r_0}{12}\varepsilon<\delta.\]
		
		Consequently, $\|g-f\|_L<\delta$, which implies that $g$ is Lipschitz. Moreover, denoting by $g_0$ the restriction of $g$ to $N$, we obtain that $g_0 \notin \overline{\SA(N,\mathbb{R})}$. We claim that $g \notin \overline{\SA(M,\mathbb{R})}$. To prove it, suppose that there exist sequences $\{h_n\}\subseteq \Lip(M,\mathbb{R})$, $\{p_n\}$, $\{q_n\} \subseteq M$ with $p_n\neq q_n$ and $h_n(m_{p_n,q_n})=\|h_n\|_L=\|g\|_L$ for every $n \in \mathbb{N}$ such that $\{\|h_n-g\|_L\}$ converges to $0$. We will distinguish three cases:
		\begin{itemize}
			\item Case 1. $p_n$, $q_n \in B(p_0,r_0)$ eventually, that is, the set $\{n \in \mathbb{N} \colon p_n, q_n \in B(p_0,r_0)\}$ is infinite.
			Recall that $B(p_0,r_0)\subseteq N$, so this assumption implies that restrictions of $h_n$ to $N$ strongly attain their norms eventually, which is impossible since $g_0 \notin \overline{\SA(N,\mathbb{R})}$.
			\item Case 2. $p_n$, $q_n \notin B(p_0,\frac{r_0}{3})$ eventually. Fix $n \in \mathbb{N}$ and note that
			\begin{align*} |(h_n-g)(m_{p_n,q_n})| &= \|g\|_L - g(m_{p_n,q_n})= \|g\|_L - \left (1-\frac{\varepsilon r_0}{12}\right )f(m_{p_n,q_n}) \\
				&\geq \|g\|_L - \left (1-\frac{\varepsilon r_0}{12}\right )\|f\|_L
				\geq \|f\|_L-\left (1-\frac{\varepsilon r_0}{12}\right )\|f\|_L= \frac{\varepsilon r_0}{12}.
			\end{align*}
			Therefore, $\|h_n-g\|_L\geq \frac{\varepsilon r_0}{12}$ for every $n \in \mathbb{N}$, which contradicts the assumption.
			\item Case 3. $p_n \in B(p_0,\frac{r_0}{3})$, $q_n \notin B(p_0,r_0)$ eventually. In this case, for a fixed $n \in \mathbb{N}$ we have that
			\begin{align*} g(m_{p_n,q_n})&=\frac{g(p_n)-g(q_n)}{d(p_n,q_n)}=\frac{\varphi(p_n)f(p_n)-(1-\frac{\varepsilon r_0}{12})f(q_n)}{d(p_n,q_n)}\\
				&\leq \frac{\varphi(p_n)f(p_n)-(1-\frac{\varepsilon r_0}{12})(f(p_n)-d(p_n,q_n))}{d(p_n,q_n)}=\frac{\left (\varphi(p_n)-\left (1-\frac{\varepsilon r_0}{12}\right )\right )f(p_n)}{d(p_n,q_n)} + \left( 1-\frac{\varepsilon r_0}{12}\right )\\
				&\leq \frac{\left (1-\left (1-\frac{\varepsilon r_0}{12}\right )\right )\frac{r_0}{3}}{d(p_n,q_n)} + \left( 1-\frac{\varepsilon r_0}{12}\right )
				\leq \frac{ \frac{\varepsilon r_0}{12}\frac{r_0}{3}}{\frac{2r_0}{3}} + \left( 1-\frac{\varepsilon r_0}{12}\right )=\frac{\varepsilon r_0}{24} + \left( 1-\frac{\varepsilon r_0}{12}\right )
				\leq 1-\frac{\varepsilon r_0}{24}.
			\end{align*}
			Since $\|g\|_L\geq 1$, we conclude that $\{h_n\}$ cannot converge to $g$ in this case either. The case when $p_n \notin B(p_0,r_0)$ and $q_n \in B(p_0,\frac{r_0}{3})$ is analogous.\qedhere
		\end{itemize}
	\end{proof}
	
	We are ready to present the proof of the main result of this section.
	
	\begin{proof}[Proof of Theorem \ref{theo:C1}]
		Since $\alpha'$ is continuous and nonidentically zero, we can take $J_0$ a nontrivial subinterval of $J$ so that $\alpha'(t)\neq 0$ for every $t \in J_0$. Hence, we can reparametrize $\alpha$ in $J_0$ with respect to arc length. Moreover, $J_0$ can be taken small enough so that $|J_0|<1$ and (\ref{equivdist}) is satisfied. Up to a change of variables we can write $J_0=[0,\rho]$ for some $0<\rho<1$.
		Consider the function $g_0\colon [0,\rho] \longrightarrow \mathbb{R}$ given by
		\[ g_0(x)=\inf \left \{\frac{\|\alpha(t)-\alpha(s)\|}{|t-s|} \colon t,s \in [0,\rho], |t-s|=x\right \} \quad \forall \, x \in [0,\rho].\]
		Let $g \colon [0,1] \longrightarrow \mathbb{R}$ be the extension of $g_0$ given by $g(x)=g(\rho)$ for every $\rho\leq x\leq 1$. In view of Lemma \ref{lemma:inf} and (\ref{equivdist}), the function $g$ satisfies statement (i) of Theorem \ref{theo:Cantor}. Consequently, there exists a Cantor set $C\subseteq[0,1]$ satisfying that $|C|>0$ and 
		\[\frac{|C\cap I|}{|I|} <g(|I|) \quad \mbox{ for every nontrivial interval } I \subseteq [0,1].\] 
		Now, write $\Gamma_0=\{\alpha(t)\colon t \in [0,\rho]\}$ and define $F \colon \Gamma_0 \longrightarrow \mathbb{R}$ by
		\[ F(\alpha(t))=\int_0^ t \chi_C(s) \, ds \quad \forall \, t \in [0,\rho].\]
		We claim that $F$ is a Lispchitz function that does not attain its Lipschitz constant and $\|F\|_L=1$. First, Lemma \ref{lemma:s} and Lemma \ref{lemma:inf} give that for a positive sequence $\{x_n\}_{n \in \mathbb{N}}$ converging to $0$ we have $\lim_{n\to \infty} |F(m_{\alpha(0),\alpha(x_n)})| =1$, which implies that $\|F\|_L\geq 1$. On the other hand, for $\alpha(t)$, $\alpha(s) \in \Gamma_0$ with $t>s$, we have that
		\[ \frac{|F(\alpha(t))-F(\alpha(s))|}{\|\alpha(t)-\alpha(s)\|}=\frac{|C\cap [s,t]|}{|t-s|}\frac{|t-s|}{\|\alpha(t)-\alpha(s)\|}<g(|t-s|) \frac{|t-s|}{\|\alpha(t)-\alpha(s)\|}\leq 1.\]
		Hence, $F$ does not attain its Lipschitz constant and $\|F\|_L=1$. Next, define $H\colon \Gamma_0 \longrightarrow \mathbb{R}$ by
		\[ H(\alpha(t))=\int_{0}^{t} \chi_C(s) - \frac{1}{8}\chi_{I\setminus C}(s)\, ds \quad \forall \, t \in [0,\rho].\]
		A similar argument as before shows that $\|H\|_L=1$.
		We claim that $H\notin \overline{\SA(\Gamma_0,\mathbb{R})}$. Indeed, if the claim is not true, then we find a sequence $\{L_n\}_{n \in \mathbb{N}}\subseteq \SA(\Gamma_0,\mathbb{R})$ converging to $H$. Define $\Phi\colon \Lip(\Gamma_0,\mathbb{R}) \longrightarrow \Lip([0,\rho],\mathbb{R})$ by
		\[ \Phi(G)(t)=G(\alpha(t)) \quad \forall \, G \in \Lip(\Gamma_0,\mathbb{R}), \quad\forall\, t \in [0,\rho].\] 
		From (\ref{equivdist}) it follows that $\Phi$ is an isomorphism with $\|\Phi\|\leq 1$, and so $\{\Phi(L_n)\}_{n\in \mathbb{N}}$ converges to $\Phi(H)$. Now, recall that Lipschitz functions from $[0,\rho]$ to $\mathbb{R}$ are differentiable almost everywhere. Moreover, their Lipschitz constant is the essential supremum of its derivative. Let us write $L'_n$ and $H'$ for the derivatives of $\Phi(L_n)$ and $\Phi(H)$, respectively. Then, $H'=\chi_C - \frac{1}{8}\chi_{I\setminus C}$ almost everywhere, so $\|H'\|_\infty = 1$. Hence, we may assume that $\|L'_n\|_\infty=1$ for every $n \in \mathbb{N}$. Take $n \in \mathbb{N}$ large enough so that $\|L_n-H\|_L<\frac{1}{8}$. Then, let us show that $L_n$ cannot attain its Lipschitz constant. First, notice that $\|L_n\|_L\geq \|\Phi(L_n)\|_L=\|L'_n\|_\infty=1$. On the other hand, we have
		\[\|L_n'-H'\|_\infty=\|\Phi(L_n-H)\|_L\leq \|L_n-H\|_L<\frac{1}{8}.\]
		Since $\|L_n'\|_\infty=1$, we conclude that $\chi_C\geq L_n'$ almost everywhere in $[0,\rho]$. Pick two points $t>s \in [0,\rho]$. First, assume that $L_n(m_{\alpha(t),\alpha(s)})\geq1$. Then, we have
		\[ F(m_{\alpha(t),\alpha(s)})= \frac{\int_s^t \chi_C(x)\, dx}{\|\alpha(t)-\alpha(s)\|}\geq \frac{\int_s^t L_n'(x)\, dx}{\|\alpha(t)-\alpha(s)\|}=L_n(m_{\alpha(t),\alpha(s)})\geq1,\]
		which contradicts the fact that $F$ does not attain its Lipschitz constant and $\|F\|_L=1$. Now, observe
		\[ L_n(m_{\alpha(s),\alpha(t)})=\frac{-\int_s^t L_n'(x)\, dx}{\|\alpha(t)-\alpha(s)\|}\leq\frac{\int_s^t\frac{1}{4}\, dx}{\|\alpha(t)-\alpha(s)\|}=\frac{1}{4}\frac{t-s}{\|\alpha(t)-\alpha(s)\|}< 1.\]
		Thus, $L_n$ does not attain its Lipschitz constant. Consequently, $H\notin \overline{\SA(\Gamma_0,\mathbb{R})}$.
		
		Finally, it is clear that we can apply Lemma \ref{lemma:open} to the subset $\Gamma_0$ and the function $H$ to obtain that $\SA(\Gamma,\mathbb{R})$ is not dense in $\Lip(\Gamma,\mathbb{R})$. Indeed, if $t_0 \in (0,\rho)$ is any Lebesgue point of density of the Cantor set $C$, then we can take $p_0=\alpha(t_0)$.
	\end{proof}

	\noindent \textbf{Acknowledgment:\ } The author is very grateful to Miguel Mart\'{i}n, Fedor Nazarov, Abraham Rueda Zoca, and specially Luis Carlos Garc\'{i}a Lirola for many comments which have improved the final version of this paper.
	
	\bibliographystyle{plain}
	\bibliography{Cantor_sets_and_C1_curves}	
\end{document}